\documentclass[10pt,a4paper]{article}
\usepackage[T1]{fontenc}
\usepackage[utf8]{inputenc}
\usepackage{authblk}
\usepackage{amsmath,amssymb,amsthm,esint,bm}
\usepackage{mathrsfs}
\usepackage{bookmark}
\usepackage{amsmath}
\allowdisplaybreaks[3]

\newtheorem{definition}{Definition}[section]
\newtheorem{theorem}[definition]{Theorem}
\newtheorem{lemma}[definition]{Lemma}

\theoremstyle{remark}
\newtheorem{remark}[definition]{Remark}
\numberwithin{equation}{section}

\title{Symmetry and Monotonicity of Positive Solutions to Schr\"{o}dinger Systems with Fractional $p$-Laplacian}

\author[a]{Lingwei Ma}
\author[b]{Zhenqiu Zhang\thanks{Corresponding author.}}

\affil[a]{School of Mathematical Sciences, Nankai University, Tianjin 300071, P.R. China}
\affil[b]{School of Mathematical Sciences and LPMC, Nankai University, Tianjin, 300071, P.R. China}

\date{\today}

\usepackage{hyperref}
\begin{document}
\maketitle
\footnotetext[1]{E-mail: 1120170026@mail.nankai.edu.cn (L. Ma), zqzhang@nankai.edu.cn (Z. Zhang).}
\begin{abstract}
In this paper, we first establish a narrow region principle and a decay at infinity theorem to extend the direct method of moving planes for general fractional $p$-Laplacian systems. By virtue of this method, we can investigate the qualitative properties of the following Schr\"{o}dinger system with fractional $p$-Laplacian
\begin{equation*}
\left\{\begin{array}{r@{\ \ }c@{\ \ }ll}
\left(-\Delta\right)_{p}^{s}u+au^{p-1}& =&f(u,v), \\[0.05cm]
\left(-\Delta\right)_{p}^{t}v+bv^{p-1}& =&g(u,v),
\end{array}\right.
\end{equation*}
where $0<s,\,t<1$ and $2<p<\infty$. We obtain the radial symmetry in the unit ball or the whole space $\mathbb{R}^{N}(N\geq2)$, the monotonicity in the parabolic domain and the nonexistence on the half space for positive solutions to the above system under some suitable conditions on $f$ and $g$, respectively.



\end{abstract}


\section{Introduction.}\label{section1}

In this paper, we are concerned with the Schr\"{o}dinger system as follows
\begin{equation}\label{model1}
\left\{\begin{array}{r@{\ \ }c@{\ \ }ll}
\left(-\Delta\right)_{p}^{s}u+au^{p-1}& =&f(u,v), & \mbox{in}\ \ \Omega\,, \\[0.05cm]
\left(-\Delta\right)_{p}^{t}v+bv^{p-1}& =&g(u,v),& \mbox{in}\ \ \Omega\,, \\[0.05cm]
u > 0, && v>0, & \mbox{on}\ \ \Omega\,,
\end{array}\right.
\end{equation}
where the fractional $p$-Laplacian  $\left(-\Delta\right)_{p}^{s}$ and $\left(-\Delta\right)_{p}^{t}$ are the nonlinear nonlocal pseudo differential operators of the types
\begin{equation}\label{Frac-p-Laplacian-s}
(-\Delta)_p^s u(x) := C_{N,sp} PV \int_{\mathbb{R}^N} \frac{|u(x)-u(y)|^{p-2}[u(x)-u(y)]}{|x-y|^{N+sp}} dy
\end{equation}
 and
 \begin{equation}\label{Frac-p-Laplacian-t}
(-\Delta)_p^t u(x) := C_{N,tp} PV \int_{\mathbb{R}^N} \frac{|u(x)-u(y)|^{p-2}[u(x)-u(y)]}{|x-y|^{N+tp}} dy.
\end{equation}
 Here, $PV$ stands for the Cauchy principal value, $C_{N,sp}$ and $C_{N,tp}$ are normalization positive constants, $0<s,\,t<1$ and $2<p<\infty$. The coefficients $a$ and $b$ are positive constants when $\Omega$ is a unit ball or the whole space. While $\Omega$ is the half space or an unbounded parabolic domain defined by
\begin{equation*}
  \Omega:=\left\{x=(x',x_{N})\in\mathbb{R}^{N}\mid x_{N}>|x'|^{2},\,x'=(x_{1},x_{2},...,x_{N-1})\right\},
\end{equation*}
  $a=a(x')$ and $b=b(x')$ are the functions that do not depend on $x_{N}$ and have lower bounds in $\Omega$.
Let
$$ \mathcal{L}_{sp} := \{ u \in L^{p-1}_{loc} \mid \int_{\mathbb{R}^N} \frac{|1+u(x)|^{p-1}}{1+|x|^{N+sp}} d x < \infty \}$$
and
$$ \mathcal{L}_{tp} := \{ v \in L^{p-1}_{loc} \mid \int_{\mathbb{R}^N} \frac{|1+v(x)|^{p-1}}{1+|x|^{N+tp}} d x < \infty \},$$
 we assume that
$$u \in C^{1,1}_{loc} \cap \mathcal{L}_{sp} \quad \mbox{and} \quad v \in C^{1,1}_{loc} \cap \mathcal{L}_{tp},$$ which are necessary to
guarantee the integrability of \eqref{Frac-p-Laplacian-s} and \eqref{Frac-p-Laplacian-t}. Obviously, for $p=2$
the fractional $p$-Laplacian coincides with the fractional Laplace operator, which is of particular interest in fractional quantum
mechanics for the study of particles on stochastic fields modelled by L\'{e}vy processes.
With respect to $p\neq2$, the nonlinear and nonlocal fractional $p$-Laplacian also arises in some important applications such as the non-local "Tug-of-War" game (cf. \cite{BCF1, BCF2}). In particular, Laskin \cite{La1, La2} originally proposed the fractional Schr\"{o}dinger equation that provides us with a
general point of view on the relationship between the statistical properties of the quantum mechanical path and the
structure of the fundamental equations of quantum mechanics.

During the last decade the elliptic equations and systems with
fractional Laplacian $\left(-\Delta\right)^{s}$ have enjoyed a
growing attention. To overcome the difficulty caused by the
non-locality of the fractional Laplacian, Caffarelli and Silvestre
\cite{CaSi} introduced an extension method to reduce the nonlocal
problem into a local one in higher dimensions. This method has been
applied successfully to investigate the equations with
$\left(-\Delta\right)^{s}$, a great number of related problems have
been studied extensively from then on (cf. \cite{BCPS, CZ} and the
references therein). Another effective method to handle the higher
order fractional Laplacian is the method of moving planes in
integral forms, which turns a given pseudo differential equations
into their equivalent integral equations, we refer \cite{CLO1,
CLO, CFY, Ma} for details. However, we always need to assume
$\frac{1}{2}\leq s<1$ or impose additional integrability conditions
on the solutions by using the extension method or the integral
equations method. Meanwhile, the aforementioned methods are not
applicable to other nonlinear nonlocal operators, such as the fully
nonlinear nonlocal operator and fractional $p$-Laplacian ($p\neq2$).
Recently, Chen et al. \cite{CLL} developed a direct method of moving
planes which can conquer these difficulties. Later a lot of articles
have been devoted to the investigation of various equations and
systems with fractional Laplacian by virtue of this direct method.
Among them, it is worth mentioning some works on generalizing the
direct method of moving planes to the fractional Laplacian system
(cf. \cite{LiuMa}) and the Schr\"{o}dinger system with  fractional
Laplacian (cf. \cite{Li}, \cite{QY}).

Afterwards, Chen et al. \cite{CLLg} extended this direct
method to consider the following fully nonlinear nonlocal equation
\begin{equation*}
  F_{\alpha}\left(u(x)\right):=C_{N,\alpha}\,PV\int_{\mathbb{R}^{N}}\frac{G\left(u(x)-u(y)\right)}{\left|x-y\right|^{N+\alpha}}dy=f\left(x,u\right),
\end{equation*}
where $G$ is a local Lipschitz continuous function, and the
operator $F_{\alpha}$ is non-degenerate in the sense that
\begin{equation}\label{G'}
G'(w) \geq c > 0.
\end{equation}
 Note that $F_{\alpha}$ becomes the fractional
Laplacian when $G(\cdot)$ is an identity map.

Indeed, the fractional
$p$-Laplacian we considered in this paper is a particular case of
the nonlinear nonlocal operator $F_{\alpha}(\cdot)$ for
$$\alpha= sp \; \mbox{ and } G(w) = |w|^{p-2}w,$$
which is degenerate if $p>2$ or singular if $p<2$ where the $w$ vanishes.
For simplicity, we will adopt this notation $G(\cdot)$ to denote the fractional $p$-Laplacian in what follows.
In this case,
$G'(w) = (p-1)|w|^{p-2}\geq0$, we have
$$ G'(w) \rightarrow \left\{\begin{array}{ll} 0, &  p > 2, \\
\infty, &  1<p<2,
\end{array}
\right.
$$
 as $ w \rightarrow 0$. It indicates that \eqref{G'} is not satisfied for the fractional
$p$-Laplacian. Unfortunately, the methods introduced in either \cite{CLL}
or \cite{CLLg} relies heavily on the non-degeneracy of $G(\cdot)$,
hence they cannot be applied directly to the
fractional $p$-Laplacian. That is why there have been only few papers
concerning the qualitative properties of the solutions for the fractional
$p$-Laplacian. In this respect, Chen and Li \cite{CL} established some new arguments to prove
the symmetry and monotonicity of positive solutions for the nonlinear
equations with fractional $p$-Laplacian. After Chen and Liu \cite{CL} extended their results to the fractional $p$-Laplacian system \eqref{model1} with $s=t$, $a=b=0$ in $\mathbb{R}^{N}$ or $\mathbb{R}_{+}^{N}$. Very recently, Wu and Niu
\cite{WN} established a narrow region principle to the equation
involving fractional $p$-Laplacian. In the spirit of \cite{WN}, Ma and Zhang \cite{MZ} proved the symmetry of positive solutions for the Choquard
equations involving the fractional $p$-Laplacian.

However, much less effort has been devoted to the Schr\"{o}dinger system \eqref{model1} to our knowledge.
The main purpose of this paper is to extend the direct method of moving planes for general fractional $p$-Laplacian systems by establishing a narrow region principle and a decay at infinity theorem. Then we can apply this method to derive the symmetry, monotonicity and nonexistence of positive solutions to the Schr\"{o}dinger system involving the fractional $p$-Laplacian in various domains.

Now we are in position to state our main results of this
paper as follows.

\begin{theorem}\label{Th1}
Let $u\in C^{1,1}_{loc}(\mathbb{R}^{N})\cap \mathcal{L}_{sp}\cap
C(\mathbb{R}^{N})$ and $v\in C^{1,1}_{loc}(\mathbb{R}^{N})\cap
\mathcal{L}_{tp}\cap C(\mathbb{R}^{N})$ be a positive solution pair
of
\begin{equation}\label{model1-RN}
\left\{\begin{array}{r@{\ \ }c@{\ \ }ll}
\left(-\Delta\right)_{p}^{s}u+au^{p-1}& =&f(u,v), & \mbox{in}\ \ \mathbb{R}^{N}\,, \\[0.05cm]
\left(-\Delta\right)_{p}^{t}v+bv^{p-1}& =&g(u,v),& \mbox{in}\ \ \mathbb{R}^{N}\,,
\end{array}\right.
\end{equation}
where $0<s,\,t<1$, $2<p<+\infty$, $a,\,b>0$ and $f,\, g\in C^{1}\left((0,+\infty)\times(0,+\infty),\mathbb{R}\right)$. Suppose that
\begin{itemize}
\item[(i)] $\frac{\partial f}{\partial v}>0$ and $\frac{\partial g}{\partial u}>0$ for $\forall \, u,\, v>0$;
\item[(ii)] $\frac{\partial f}{\partial u}\leq u^{m-1}v^{n}$ and $\frac{\partial f}{\partial v}\leq u^{m}v^{n-1}$ as $(u,v)\rightarrow(0^{+},0^{+})$;
\item[(iii)] $\frac{\partial g}{\partial u}\leq u^{q-1}v^{r}$ and $\frac{\partial g}{\partial v}\leq u^{q}v^{r-1}$ as $(u,v)\rightarrow(0^{+},0^{+})$;
\item[(iv)] $\frac{\partial f}{\partial u}-a(p-1)u^{p-2}$ is increasing with respect to $u$ as $u\rightarrow 0^{+}$ and $\frac{\partial g}{\partial v}-b(p-1)v^{p-2}$ is increasing with respect to $v$ as $v\rightarrow 0^{+}$;
\item[(v)]  $u(x) \sim \frac{1}{|x|^\gamma}$ and $ v(x) \sim \frac{1}{|x|^\tau}$  as $|x|\rightarrow\infty$,
\end{itemize}
where $m,\,r \geq p-1$, $n,\,q \geq1$ and $\gamma,\,\tau>0$ satisfy
 \begin{equation}\label{parameter1}
 \min\{\gamma (m-1)+\tau n,\,\gamma m+\tau(n-1)\}>\gamma(p-2)+sp
 \end{equation}
and
\begin{equation}\label{parameter2}
  \min\{\tau (r-1)+\gamma q,\,\tau r+\gamma(q-1)\}>\tau(p-2)+tp.
\end{equation}
Then $u$ and $v$ are radially symmetric and monotone decreasing about some point in $\mathbb{R}^n$.
\end{theorem}

\begin{remark}
Due to the presence of the fractional
$p$-Laplacian and $a,\,b\neq0$, the Kelvin transform is no longer valid, so we need to impose the additional assumptions on
the behavior of $u$ and $v$ at infinity.
\end{remark}

\begin{theorem}\label{Th2}
Let $u\in C^{1,1}_{loc}(B_{1}(0))\cap \mathcal{L}_{sp}\cap
C(B_{1}(0))$ and $v\in C^{1,1}_{loc}(B_{1}(0))\cap
\mathcal{L}_{tp}\cap C(B_{1}(0))$ be a positive solution pair of
\begin{equation}\label{model1-B1}
\left\{\begin{array}{r@{\ \ }c@{\ \ }ll}
\left(-\Delta\right)_{p}^{s}u+au^{p-1}& =&f(u,v), & x\in B_{1}(0)\,, \\[0.05cm]
\left(-\Delta\right)_{p}^{t}v+bv^{p-1}& =&g(u,v),& x\in B_{1}(0)\,, \\[0.05cm]
u = v=0,&& & x\not\in B_{1}(0)\,,
\end{array}\right.
\end{equation}
where $0<s,\,t<1$, $2<p<+\infty$ and $a,\,b>0$. Suppose that $f,\,g\in C^{0,1}\left([0,+\infty)\times[0,+\infty),\mathbb{R}\right)$ satisfy
\begin{equation}\label{f}
  f(u,v_{1})<f(u,v_{2}) \quad \mbox{for}\,\, \forall\, u\geq0,\,0\leq v_{1}<v_{2}
\end{equation}
and
\begin{equation}\label{g}
  g(u_{1},v)<g(u_{2},v) \quad \mbox{for}\,\, \forall\, v\geq0,\,0\leq u_{1}<u_{2},
\end{equation}
respectively. Then $u$ and $v$ are radially symmetric and monotone decreasing about the origin.
\end{theorem}

\begin{theorem}\label{Th3}
Let $u\in C^{1,1}_{loc}(\Omega)\cap \mathcal{L}_{sp}\cap C(\Omega)$
and $v\in C^{1,1}_{loc}(\Omega)\cap \mathcal{L}_{tp}\cap C(\Omega)$
be a positive solution pair of
\begin{equation}\label{model1-Pd}
\left\{\begin{array}{r@{\ \ }c@{\ \ }ll}
\left(-\Delta\right)_{p}^{s}u+a(x')u^{p-1}& =&f(u,v), & x\in \Omega\,, \\[0.05cm]
\left(-\Delta\right)_{p}^{t}v+b(x')v^{p-1}& =&g(u,v),& x\in \Omega\,, \\[0.05cm]
u = v=0,&& & x\not\in \Omega\,,
\end{array}\right.
\end{equation}
where $0<s,\,t<1$, $2<p<+\infty$ and $a(x')$, $b(x')$ are bounded
from below in $\Omega$. Meanwhile, $f,\,g\in
C^{0,1}\left([0,+\infty)\times[0,+\infty),\mathbb{R}\right)$ satisfy
\eqref{f} and \eqref{g}. Then $u$ and $v$ are strictly increasing
with respect to the $x_{N}$-axis.
\end{theorem}

\begin{theorem}\label{Th4}
Let $u\in C^{1,1}_{loc}(\mathbb{R}_{+}^{N})\cap \mathcal{L}_{sp}\cap
C(\mathbb{R}_{+}^{N})$ and $v\in
C^{1,1}_{loc}(\mathbb{R}_{+}^{N})\cap \mathcal{L}_{tp}\cap
C(\mathbb{R}_{+}^{N})$ be a nonnegative solution pair of
\begin{equation}\label{model1-R+N}
\left\{\begin{array}{r@{\ \ }c@{\ \ }ll}
\left(-\Delta\right)_{p}^{s}u+a(x')u^{p-1}& =&f(u,v), & x\in \mathbb{R}_{+}^{N}\,, \\[0.05cm]
\left(-\Delta\right)_{p}^{t}v+b(x')v^{p-1}& =&g(u,v),& x\in \mathbb{R}_{+}^{N}\,, \\[0.05cm]
u = v=0,&& & x\not\in \mathbb{R}_{+}^{N}\,,
\end{array}\right.
\end{equation}
where $0<s,\,t<1$, $2<p<+\infty$ and $a(x')$, $b(x')$ are bounded from below in $\Omega$. Meanwhile, $f,\,g\in C^{0,1}\left([0,+\infty)\times[0,+\infty),\mathbb{R}\right)$ satisfy
\eqref{f}, \eqref{g} and
\begin{equation}\label{R+-fg}
  f(0,0)=g(0,0)=0.
\end{equation}
Suppose that
\begin{equation}\label{decay-uv}
  \underset{|x| \rightarrow \infty}{\underline{\lim}}u(x)=\underset{|x| \rightarrow \infty}{\underline{\lim}}v(x)=0,
\end{equation}
then $u(x)=v(x)\equiv0$ in $\mathbb{R}^{N}$.
\end{theorem}
The remainder of this paper is organized as follows. In section
\ref{section2}\,, we establish the the corresponding
narrow region principle and decay at infinity theorem. Section \ref{section3} contains the proof of
Theorem \ref{Th1} and \ref{Th2}\,. Moreover, Theorem \ref{Th3} and \ref{Th4} are proved in the last section.

\section{Narrow Region Principle and Decay at Infinity }\label{section2}

In this section, we construct the
narrow region principle and the decay at infinity theorem for anti-symmetric functions, which play essential roles in carrying on the
 direct method of moving planes for the fractional $p$-Laplacian systems.

Before establishing two maximum principles, we first introduce the following notations to facilitate our description.
Taking the whole space $\mathbb{R}^{N}$ as an example. Let
$$T_{\lambda} :=\{x \in \mathbb{R}^{N}\mid x_1=\lambda, \mbox{ for some } \lambda\in \mathbb{R}\}$$
be the moving planes,
$$\Sigma_{\lambda} :=\{x \in \mathbb{R}^{N}\mid x_1<\lambda\}$$
be the region to the left of $T_{\lambda}$ and
$$ x^{\lambda} :=(2\lambda-x_1, x_2, ..., x_N)$$
be the reflection of $x$ with respect to $T_{\lambda}$.
Let $\left(u,v\right)$ be a solution pair of Schr\"{o}dinger system \eqref{model1-RN},
we denote the reflected functions by $u_{\lambda}(x):=u(x^{\lambda})$ and $v_{\lambda}(x):=v(x^{\lambda})$.
Moreover,
\begin{equation*}
\left\{\begin{array}{r@{\ \ }c@{\ \ }ll}
U_{\lambda} (x)& :=&u(x^{\lambda}) - u(x), \\[0.05cm]
V_{\lambda} (x)& :=&v(x^{\lambda}) - v(x),
\end{array}\right.
\end{equation*}
represent the comparison between the values of $u(x)$, $u(x^{\lambda})$ and $v(x)$, $v(x^{\lambda})$, respectively.
Evidently, $U_{\lambda}$ and $V_{\lambda}$ are anti-symmetric
functions, i.e., $U_{\lambda}(x^{\lambda})=-U_{\lambda}(x)$ and
$V_{\lambda}(x^{\lambda})=-V_{\lambda}(x)$. From now on,
$C$ denotes a constant whose value may be different from line to
line, and only the relevant dependence is specified in what follows.

Now we start by establishing the following narrow region principle, which generalizes Theorem 1.1 in \cite{WN} to the
 fractional $p$-Laplacian systems.

\begin{theorem}\label{NRP}(\textbf{Narrow region principle})
Let $\Omega$ be a bounded narrow region in $\Sigma_{\lambda}$, such
that it is contained in $\left\{x\mid
\lambda-\delta<x_{1}<\lambda\right\}$ with a small $\delta$. Assume
that $u\in \mathcal{L}_{sp}\cap C^{1,1}_{loc}(\Sigma_{\lambda})$ and
$v\in \mathcal{L}_{tp}\cap C^{1,1}_{loc}(\Sigma_{\lambda})$,
$U_{\lambda}$, $V_{\lambda}$ are lower semi-continuous on
$\overline{\Omega}$ and satisfy
\begin{equation}\label{NRP-model}
\left\{\begin{array}{r@{\ \ }c@{\ \ }ll}
&&\left(-\Delta\right)_{p}^{s}u_{\lambda}(x)-\left(-\Delta\right)_{p}^{s}u(x)+C_{1}(x)U_{\lambda}(x)+C_{2}(x)V_{\lambda}(x)\geq0, & \ \ x\in\Omega\,, \\[0.05cm]
&&\left(-\Delta\right)_{p}^{t}v_{\lambda}(x)-\left(-\Delta\right)_{p}^{t}v(x)+C_{3}(x)U_{\lambda}(x)+C_{4}(x)V_{\lambda}(x)\geq0, & \ \ x\in\Omega\,, \\[0.05cm]
&&U_{\lambda}(x) \geq 0,\,  V_{\lambda}(x)\geq0, & \ \ x\in\Sigma_{\lambda}\backslash\Omega\,, \\[0.05cm]
&&U_{\lambda}(x^{\lambda})=-U_{\lambda}(x),\,  V_{\lambda}(x^{\lambda})=-V_{\lambda}(x), & \ \ x\in\Sigma_{\lambda}\,,
\end{array}\right.
\end{equation}
where $C_{1}(x)$, $C_{2}(x)$, $C_{3}(x)$ and $C_{4}(x)$ have lower bounds as $C_{1},\,C_{2},\,C_{3},\,C_{4}\in\mathbb{R}$, respectively,
and $C_{2}(x),\,C_{3}(x)<0$ in $\Omega$. If there exist
$y^{0},\,y^{1}\in\Sigma_{\lambda}$ such that $U_{\lambda}(y^{0})>0$ and $V_{\lambda}(y^{1})>0$, then
\begin{equation}\label{NRP-r1}
  U_{\lambda}(x),\,V_{\lambda}(x)\geq0,\quad x\in\Omega
\end{equation}
for sufficiently small $\delta$.
Moreover, if $U_{\lambda}(x)=0$ or $V_{\lambda}(x)=0$ at some point in $\Omega$, then
\begin{equation}\label{NRP-r2}
  U_{\lambda}(x)=V_{\lambda}(x)\equiv0 \quad \mbox{in} \
  \mathbb{R}^{N}.
\end{equation}
The above conclusions are valid for an unbounded narrow region $\Omega$ if we further suppose that
\begin{equation*}
\underset{|x| \rightarrow \infty}{\underline{\lim}} U_{\lambda}(x),\,V_{\lambda}(x)\geq0.
\end{equation*}
\end{theorem}

\begin{remark}
Compared with the narrow region principle for the Schr\"{o}dinger system
with fractional Laplace equations in \cite{QY}, here we need to impose
the extra assumption that there exist
$y^{0},\,y^{1}\in\Sigma_{\lambda}$ such that $U_{\lambda}(y^{0})>0$ and $V_{\lambda}(y^{1})>0$, respectively. As a matter of fact, this condition is automatically satisfied for \eqref{model1-RN}, \eqref{model1-B1}, \eqref{model1-Pd} and \eqref{model1-R+N}.
\end{remark}

\begin{proof}[Proof of Theorem \ref{NRP}] The proof goes by contradiction. Without loss of generality, we assume that there exists $x^{0}\in \Omega$ such that
$$U_{\lambda}(x^{0})=\min_{\Omega} U_{\lambda}<0.$$
Otherwise, the same arguments as follows can also yield a contradiction for the case that there exists $x^{1}\in \Omega$ such that
$V_{\lambda}(x^{1})=\displaystyle\min_{\Omega} V_{\lambda}<0$.

By a direct calculation, we obtain
\begin{eqnarray}\label{n-1}
  &&(-\Delta)_{p}^{s}u_{\lambda}(x^{0})-(-\Delta)_{p}^{s}u(x^{0})\nonumber\\
  &=& C_{N,sp}\,PV\int_{\mathbb{R}^{N}}\frac{G(u_{\lambda}(x^{0})-u_{\lambda}(y))-G(u(x^{0})-u(y))}{|x^{0}-y|^{N+sp}}dy \nonumber\\
   &=& C_{N,sp}\,PV\int_{\Sigma_{\lambda}}\frac{G(u_{\lambda}(x^{0})-u_{\lambda}(y))}{|x^{0}-y|^{N+sp}}
   +\frac{G(u_{\lambda}(x^{0})-u(y))}{|x^{0}-y^{\lambda}|^{N+sp}}dy \nonumber \\
   && -C_{N,sp}\,PV\int_{\Sigma_{\lambda}}\frac{G(u(x^{0})-u(y))}{|x^{0}-y|^{N+sp}}
   +\frac{G(u(x^{0})-u_{\lambda}(y))}{|x^{0}-y^{\lambda}|^{N+sp}}dy \nonumber  \\
   &=&  C_{N,sp}\int_{\Sigma_{\lambda}}\frac{\left[G(u_{\lambda}(x^{0})-u_{\lambda}(y))-G(u(x^{0})-u_{\lambda}(y))\right]
   +\left[G(u_{\lambda}(x^{0})-u(y))-G(u(x^{0})-u(y))\right]}{|x^{0}-y^{\lambda}|^{N+sp}}dy \nonumber\\
   &&+C_{N,sp}\,PV\int_{\Sigma_{\lambda}}\left[\frac{1}{|x^{0}-y|^{N+sp}}-\frac{1}{|x^{0}-y^{\lambda}|^{N+sp}}\right]
   \left[G(u_{\lambda}(x^{0})-u_{\lambda}(y))-G(u(x^{0})-u(y))\right]dy \nonumber\\
   &:=& C_{N,sp}\left(I_{1}+I_{2}\right).
\end{eqnarray}

We start by estimating $I_{1}$. It follows from mean value theorem and the monotonicity of $G$ that
\begin{equation}\label{n-2}
  I_{1}=U_{\lambda}(x^{0})\int_{\Sigma_{\lambda}}\frac{G'\left(\zeta(y)\right)+G'\left(\eta(y)\right)}{|x^{0}-y^{\lambda}|^{N+sp}}dy\leq0,
\end{equation}
where $\zeta(y)\in \left(u_{\lambda}(x^{0})-u_{\lambda}(y),\,u(x^{0})-u_{\lambda}(y)\right)$ and $\eta(y)\in \left(u_{\lambda}(x^{0})-u(y),\,u(x^{0})-u(y)\right)$.

Now we turn our attention to $I_{2}$. Let $\delta_{x^{0}}:=\operatorname{dist}\left\{x^{0},\,T_{\lambda}\right\}$, it is not difficult to verify that $\delta_{x^{0}}=\lambda-x_{1}^{0}$. Then applying mean value theorem again, we compute
\begin{equation}\label{n-3}
  \frac{1}{|x^{0}-y|^{n+sp}}-\frac{1}{|x^{0}-y^{\lambda}|^{N+sp}}=
  \frac{2\left(N+sp\right)\left(\lambda-y_{1}\right)}{\left|x^{0}-\varsigma\right|^{N+sp+2}}\delta_{x^{0}},
\end{equation}
where $\varsigma$ is a point on the line segment between $y$ and $y^{\lambda}$. Thus,
\begin{eqnarray}\label{n-4}
  I_{2}
   &=& \delta_{x^{0}}\int_{\Sigma_{\lambda}}\frac{2\left(N+sp\right)\left(\lambda-y_{1}\right)}{\left|x^{0}-\varsigma\right|^{N+sp+2}}
   \left[G(u_{\lambda}(x^{0})-u_{\lambda}(y))-G(u(x^{0})-u(y))\right]dy \nonumber\\
   &:=&  \delta_{x^{0}}F(x^{0}).
\end{eqnarray}

Before estimating further, we claim that there exists a positive constant $c_{1}$ such that
\begin{equation}\label{n-5}
  F(x^{0})\leq-\frac{c_{1}}{2}
\end{equation}
for sufficiently small $\delta_{x^{0}}$. In doing so, we first show that
\begin{equation}\label{n-6}
  F(x^{0})<0.
\end{equation}
Applying the monotonicity of $G$, we derive
\begin{equation*}
  G(u_{\lambda}(x^{0})-u_{\lambda}(y))-G(u(x^{0})-u(y))\leq0,
\end{equation*}
which is not identically zero in $\Sigma_{\lambda}$. Hence, we conclude \eqref{n-6} by virtue of the continuity of $u$ and
\begin{equation*}
  \frac{2\left(N+sp\right)\left(\lambda-y_{1}\right)}{\left|x^{0}-\varsigma\right|^{N+sp+2}}
  =\frac{1}{\delta_{x^{0}}}\left[\frac{1}{|x^{0}-y|^{N+sp}}-\frac{1}{|x^{0}-y^{\lambda}|^{N+sp}}\right]>0.
\end{equation*}
Next, we continue to prove \eqref{n-5}. If not, then
\begin{equation*}
  F(x^{0})\rightarrow0 \ \ \mbox{as}\ \ \delta_{x^{0}}\rightarrow 0.
\end{equation*}
It is revealed that if $\delta_{x^{0}}\rightarrow 0$, then
\begin{equation*}
  G(u_{\lambda}(x^{0})-u_{\lambda}(y))-G(u(x^{0})-u(y))\rightarrow 0 \ \ \mbox{for} \ \ \forall\, y\in\Sigma_{\lambda}.
\end{equation*}
Utilizing the monotonicity of $G$ and the continuity of $u$ again, we obtain
\begin{equation*}
  U_{\lambda}(x^{0})-U_{\lambda}(y)\rightarrow 0 \ \ \mbox{for} \ \ \forall\, y\in\Sigma_{\lambda}.
\end{equation*}
Note that $U_{\lambda}(x^{0})\rightarrow 0$ as $\delta_{x^{0}}\rightarrow 0$, then we derive
$$U_{\lambda}(y)\equiv0 \ \ \mbox{for} \ \ \forall\, y\in\Sigma_{\lambda},$$
which contradicts with the condition that there exists $y^{0}\in\Sigma_{\lambda}$ such that $U_{\lambda}(y^{0})>0$. Thus, we can deduce there exists a positive constant $c_{2}$ such that
\begin{equation*}
  F(x^{0})\rightarrow-c_{2} \ \ \mbox{as}\ \ \delta_{x^{0}}\rightarrow 0.
\end{equation*}
Hence, we conclude the assertion \eqref{n-5} from the continuity of $F(x^{0})$ with respect to $x^{0}$,.

Inserting \eqref{n-5} into \eqref{n-4}, we obtain
\begin{equation}\label{n-7}
  I_{2}\leq-\frac{c_{2}}{2}\delta_{x^{0}}.
\end{equation}
Then a combination of \eqref{n-1}, \eqref{n-2} and \eqref{n-7} yields that
\begin{equation}\label{n-8}
  (-\Delta)_{p}^{s}u_{\lambda}(x^{0})-(-\Delta)_{p}^{s}u(x^{0})\leq -C\delta_{x^{0}}.
\end{equation}
Thus, applying the first inequality in \eqref{NRP-model} and $C_{1}(x)\geq C_{1}$, we derive
\begin{eqnarray}\label{n-9}
  -C_{2}(x^{0})V_{\lambda}(x^{0}) &\leq& -C\delta_{x^{0}}+ C_{1}(x^{0})U_{\lambda}(x^{0})\nonumber\\
   &\leq&  -C\delta_{x^{0}}+ C_{1}U_{\lambda}(x^{0}).
\end{eqnarray}
Note that since
\begin{equation*}
  \nabla U_{\lambda}(x^{0})=0,
\end{equation*}
 we get
\begin{equation*}
  0=U_{\lambda}(x^{2})=U_{\lambda}(x^{0})+\nabla U_{\lambda}(x^{0})\left(x^{2}-x^{0}\right)+o\left(|x^{2}-x^{0}|\right),
\end{equation*}
by Taylor expansion, where
$x^{2}=\left(\lambda,x_{2}^{0},...,x_{N}^{0}\right)\in T_{\lambda}$.
Hence, it means that
\begin{equation}\label{n-10}
  U_{\lambda}(x^{0})=o(1)\delta_{x^{0}}
\end{equation}
for sufficiently small $\delta_{x^{0}}$. Substituting \eqref{n-10} into \eqref{n-9}, we have
\begin{equation*}
  -C_{2}(x^{0})V_{\lambda}(x^{0})\leq \delta_{x^{0}}\left(-C+C_{1}o(1)\right)<0,
\end{equation*}
for small enough $\delta_{x^{0}}$. Then it follows from $C_{2}(x)<0$
that $V_{\lambda}(x^{0})<0$. Hence, the lower semi-continuity of
$V_{\lambda}$ on $\overline{\Omega}$ implies there exists
$x^{1}\in\Omega$ such that
\begin{equation*}
   V_{\lambda}(x^{1})=\min_{\Omega} V_{\lambda}<0.
\end{equation*}
In analogy with \eqref{n-8} and \eqref{n-10}, we can deduce
\begin{equation}\label{n-11}
  (-\Delta)_{p}^{t}v_{\lambda}(x^{1})-(-\Delta)_{p}^{t}v(x^{1})\leq -C\delta_{x^{1}}
\end{equation}
and
\begin{equation}\label{n-12}
  V_{\lambda}(x^{1})=o(1)\delta_{x^{1}}
\end{equation}
for sufficiently small $\delta_{x^{1}}$, respectively, where $\delta_{x^{1}}:=\operatorname{dist}\left\{x^{1},\,T_{\lambda}\right\}=\lambda-x_{1}^{1}$.

In terms of the assumptions imposed on $C_{3}(x)$ and $C_{4}(x)$ in Theorem \ref{NRP}\,, and combining the second inequality in \eqref{NRP-model}, \eqref{n-11}, \eqref{n-10} with \eqref{n-12}, we can conclude that
\begin{eqnarray*}
   0&\leq& \left(-\Delta\right)_{p}^{t}v_{\lambda}(x^{1})-\left(-\Delta\right)_{p}^{t}v(x^{1})+C_{3}(x^{1})U_{\lambda}(x^{1})
   +C_{4}(x^{1})V_{\lambda}(x^{1})\nonumber \\
   &\leq& -C\delta_{x^{1}}+C_{3}U_{\lambda}(x^{0})
   +C_{4}V_{\lambda}(x^{1})\nonumber\\
   &=& -C\delta_{x^{1}}+C_{3}\,o(1)\delta_{x^{0}}
   +C_{4}\,o(1)\delta_{x^{1}}<0
\end{eqnarray*}
for sufficiently small $\delta$, which deduces a contradiction. Thus, \eqref{NRP-r1} is proved.

Subsequently, in order to prove \eqref{NRP-r2}, we assume that there exists a point $\widetilde{x}\in\Omega$ such that
$$U_{\lambda}(\widetilde{x})=\min_{\Sigma_{\lambda}} U_{\lambda}=0.$$
Now we claim that
\begin{equation}\label{n-13}
 U_{\lambda}(x) \equiv 0 , \;\; x \in \Sigma_{\lambda}.
\end{equation}
If not, then
\begin{eqnarray}\label{n-14}
  &&(-\Delta)_{p}^{s}u_{\lambda}(\widetilde{x})-(-\Delta)_{p}^{s}u(\widetilde{x})\nonumber\\
   &=&  C_{N,sp}\,PV\int_{\mathbb{R}^{N}}\frac{G(u_{\lambda}(\widetilde{x})-u_{\lambda}(y))-G(u(\widetilde{x})-u(y))
   }{|\widetilde{x}-y|^{N+sp}}dy \nonumber\\
   &=& C_{N,sp}\,PV\int_{\mathbb{R}^{N}}\frac{G(u(\widetilde{x})-u_{\lambda}(y))-G(u(\widetilde{x})-u(y))
   }{|\widetilde{x}-y|^{N+sp}}dy\nonumber\\
   &=& C_{N,sp}\,PV\int_{\Sigma_{\lambda}}\left[\frac{1}{|\widetilde{x}-y|^{N+sp}}-\frac{1}{|\widetilde{x}-y^{\lambda}|^{N+sp}}\right]
   \left[G(u(\widetilde{x})-u_{\lambda}(y))-G(u(\widetilde{x})-u(y))\right]dy \nonumber\\
   &<&0.
\end{eqnarray}
Combining the above inequality with \eqref{NRP-model} and $C_{2}(x)<0$, we derive
$$V_{\lambda}(\widetilde{x})<0$$
which is contradictive with \eqref{NRP-r1}.
 Thus, it follows from \eqref{n-13} and the anti-symmetry of $U_{\lambda}(x)$ that
\begin{equation}\label{n-15}
 U_{\lambda}(x)\equiv0 \quad \mbox{in} \
  \mathbb{R}^{n}.
\end{equation}
Applying \eqref{n-15}, \eqref{NRP-model} and \eqref{NRP-r1}, we obtain
\begin{equation*}
 V_{\lambda}(x)\equiv0 \quad \mbox{ in} \
  \Omega.
\end{equation*}
It remains to be proved $V_{\lambda}(x)\equiv0$ for $x\in\Sigma_{\lambda}\backslash\Omega$.
If not, the same argument as \eqref{n-14} deduces that
\begin{equation*}
  (-\Delta)_{p}^{t}v_{\lambda}(x)-(-\Delta)_{p}^{t}v(x)<0,
\end{equation*}
which is contradictive with \eqref{NRP-r1} by the second inequality in \eqref{NRP-model} and $C_{3}(x)<0$. A combination of $V_{\lambda}(x)\equiv0$ in $\Sigma_{\lambda}$ and the anti-symmetry of $V_{\lambda}(x)$ yields that
\begin{equation*}
 V_{\lambda}(x)\equiv0 \quad \mbox{ in} \
  \mathbb{R}^{n}.
\end{equation*}
Similarly, one can show that if $V_{\lambda}(x)=0$ at some point in $\Omega$, then both $U_{\lambda}(x)$ and $V_{\lambda}(x)$ are identically zero in $\mathbb{R}^{n}$.

For the unbounded narrow region $\Omega$, the condition
\begin{equation*}
\underset{|x| \rightarrow \infty}{\underline{\lim}} U_{\lambda}(x),\,V_{\lambda}(x)\geq0
\end{equation*}
guarantees that the negative minimum of $U_{\lambda}$ and $V_{\lambda}$ must be attained at some point $x^{0}$ and $x^{1}$, respectively, then we can derive the similar contradictions as above.

This completes the proof of Theorem \ref{NRP} .
\end{proof}

Furthermore, in order to carry on the direct method of moving planes in $\mathbb{R}^{N}$, we also need to construct the decay at infinity theorem.
We proceed by introducing the following useful technical lemma.
\begin{lemma}\label{Tech-Lem}{\rm (cf. \cite{CL})}
For $G(w)=|w|^{p-2}w$, it follows from mean value theorem that
$$G(w_{2})-G(w_{1})=G'(\zeta)(w_{2}-w_{1}).$$
Then there exists a positive constant $c_{0}$ such that
\begin{equation}\label{L-1}
  |\zeta|\geq c_{0}\max\left\{|w_{1}|,|w_{2}|\right\}.
\end{equation}
\end{lemma}

Now we turn to establish the decay at infinity theorem for the fractional $p$-Laplacian systems, which is important for the proof of Theorem \ref{Th1}\,.
\begin{theorem}\label{Decay}(\textbf{Decay at infinity})
Let $\Omega$ be an unbounded region in $\Sigma_{\lambda}$. Assume
that $u\in \mathcal{L}_{sp}\cap C^{1,1}_{loc}(\Omega)$ and $v\in
\mathcal{L}_{tp}\cap C^{1,1}_{loc}(\Omega)$, $U_{\lambda}$,
$V_{\lambda}$ are lower semi-continuous on $\overline{\Omega}$ and
satisfy
\begin{equation}\label{decay-model}
\left\{\begin{array}{r@{\ \ }c@{\ \ }ll}
&&\left(-\Delta\right)_{p}^{s}u_{\lambda}(x)-\left(-\Delta\right)_{p}^{s}u(x)+C_{1}(x)U_{\lambda}(x)+C_{2}(x)V_{\lambda}(x)\geq0, & \ \ x\in\Omega\,, \\[0.05cm]
&&\left(-\Delta\right)_{p}^{t}v_{\lambda}(x)-\left(-\Delta\right)_{p}^{t}v(x)+C_{3}(x)U_{\lambda}(x)+C_{4}(x)V_{\lambda}(x)\geq 0, & \ \ x\in\Omega\,, \\[0.05cm]
&&U_{\lambda}(x) \geq 0,\,  V_{\lambda}(x)\geq0, & \ \ x\in\Sigma_{\lambda}\backslash\Omega\,, \\[0.05cm]
&&U_{\lambda}(x^{\lambda})=-U_{\lambda}(x),\,  V_{\lambda}(x^{\lambda})=-V_{\lambda}(x), & \ \ x\in\Sigma_{\lambda}\,,
\end{array}\right.
\end{equation}
where $C_{1}(x),\,C_{4}(x)\geq0$ and $C_{2}(x),\,C_{3}(x)<0$ on $\Omega$ such that
\begin{equation}\label{decay-C23}
 \underset{|x| \rightarrow \infty}{\underline{\lim}}C_{2}(x)|x|^{\gamma(p-2)+sp}=0,\,
\underset{|x| \rightarrow \infty}{\underline{\lim}} C_{3}(x)|x|^{\tau(p-2)+tp}=0,
\end{equation}
where $\gamma$ and $\tau$ given in Theorem \ref{Th1}\,.
Then there exists a positive constant $R_{0}$ such that if
\begin{equation}\label{decay-UV}
  U_{\lambda}(x^{0})=\min_{\Omega} U_{\lambda}<0,\quad  V_{\lambda}(x^{1})=\min_{\Omega} V_{\lambda}<0,
\end{equation}
then at least one of $x^{0}$ and $x^{1}$ satisfies
\begin{equation}\label{decay-result}
  |x|\leq R_{0}.
\end{equation}
\end{theorem}

\begin{proof}
The proof is carried out by contradiction. If \eqref{decay-result} is violated, then by the monotonicity of $G$ and mean value theorem, we can compute
\begin{eqnarray}\label{d-1}
  &&(-\Delta)_{p}^{s}u_{\lambda}(x^{0})-(-\Delta)_{p}^{s}u(x^{0})\nonumber\\
   &=& C_{N,sp}\,PV\int_{\Sigma_{\lambda}}\left[\frac{1}{|x^{0}-y|^{N+sp}}-\frac{1}{|x^{0}-y^{\lambda}|^{N+sp}}\right]
   \left[G(u_{\lambda}(x^{0})-u_{\lambda}(y))-G(u(x^{0})-u(y))\right]dy \nonumber\\
   && + C_{N,sp}\int_{\Sigma_{\lambda}}\frac{\left[G(u_{\lambda}(x^{0})-u_{\lambda}(y))-G(u(x^{0})-u_{\lambda}(y))\right]
   +\left[G(u_{\lambda}(x^{0})-u(y))-G(u(x^{0})-u(y))\right]}{|x^{0}-y^{\lambda}|^{N+sp}}dy\nonumber\\
   &\leq& C_{N,sp}\int_{\Sigma_{\lambda}}\frac{\left[G(u_{\lambda}(x^{0})-u_{\lambda}(y))-G(u(x^{0})-u_{\lambda}(y))\right]
   +\left[G(u_{\lambda}(x^{0})-u(y))-G(u(x^{0})-u(y))\right]}{|x^{0}-y^{\lambda}|^{N+sp}}dy\nonumber\\
   &=& C_{N,sp}\,U_{\lambda}(x^{0})\int_{\Sigma_{\lambda}}\frac{G'\left(\zeta(y)\right)+G'\left(\eta(y)\right)}{|x^{0}-y^{\lambda}|^{N+sp}}dy,
\end{eqnarray}where $\zeta(y)\in \left(u_{\lambda}(x^{0})-u_{\lambda}(y),\,u(x^{0})-u_{\lambda}(y)\right)$ and $\eta(y)\in \left(u_{\lambda}(x^{0})-u(y),\,u(x^{0})-u(y)\right)$.
Let $R=|x^0|$ and $x_{R}^{\lambda}=(x_{1}^{0}+(M+1)|x^{0}|,x_{2}^{0},...,x_{N}^{0})$, then $|x^{\lambda}_R|\geq M R$.
Here $M$ is a sufficiently large number such that
$$B_R(x_R) \subset \Sigma_\lambda \, \mbox{and}\, B_R(x_R^{\lambda}) \subset \Sigma_\lambda^{C}$$
for fixed $\lambda$. Moreover, the $M$ guarantees that
\begin{equation}\label{d-2}
 u(y) \leq \frac{C}{M^\gamma R^\gamma} \leq \frac{c}{R^{\gamma}} \leq u(x^{0})
\end{equation}
for any $y \in B_R(x_R^{\lambda})$ by (v) in Theorem \ref{Th1}\,.
Hence, a combination of \eqref{d-1}, \eqref{d-2} and Lemma \ref{Tech-Lem} yields that
\begin{eqnarray*}
  && C_{N,sp}\,U_{\lambda}(x^{0})\int_{\Sigma_{\lambda}}\frac{G'\left(\zeta(y)\right)+G'\left(\eta(y)\right)}{|x^{0}-y^{\lambda}|^{N+sp}}dy\\
   &\leq&  C_{N,sp}\,U_{\lambda}(x^{0})\int_{B_R(x_R)}\frac{G'\left(\zeta(y)\right)}{|x^{0}-y^{\lambda}|^{N+sp}}dy \\
   &=&  C_{N,sp}\,(p-1)U_{\lambda}(x^{0})\int_{B_R(x_R)}\frac{\left|\zeta(y)\right|^{p-2}}{|x^{0}-y^{\lambda}|^{N+sp}}dy \\
   &\leq& C_{N,sp}\,c_{0}^{p-2}(p-1)U_{\lambda}(x^{0})\int_{B_R(x_R)}\frac{\left|u(x^{0})-u_{\lambda}(y)\right|^{p-2}}{|x^{0}-y^{\lambda}|^{N+sp}}dy \\
  &\leq& C_{N,sp}\,c_{0}^{p-2}(p-1)U_{\lambda}(x^{0})\int_{B_R(x_R^{\lambda})}\frac{\left|u(x^{0})-\frac{C}{M^{\gamma}c}u(x^{0})\right|^{p-2}}{|x^{0}-y|^{N+sp}}dy \\
   &\leq& C \,U_{\lambda}(x^{0})\int_{B_R(x_R^{\lambda})}\frac{u^{p-2}(x^{0})}{|x^{0}-y|^{N+sp}}dy\\
   &\leq& C \,\frac{U_{\lambda}(x^{0})}{R^{\gamma(p-2)+sp}}.
\end{eqnarray*}
That is to say,
\begin{equation}\label{d-3}
  (-\Delta)_{p}^{s}u_{\lambda}(x^{0})-(-\Delta)_{p}^{s}u(x^{0})\leq C \,\frac{U_{\lambda}(x^{0})}{|x^{0}|^{\gamma(p-2)+sp}}.
\end{equation}
Applying the first inequality in \eqref{decay-model} and $C_{1}(x)\geq0$, we derive
\begin{equation}\label{d-4}
  U_{\lambda}(x^{0})\geq-CC_{2}(x^{0})|x^{0}|^{\gamma(p-2)+sp}V_{\lambda}(x^{0}).
\end{equation}
Then it follows from $C_{2}(x)<0$ that
\begin{equation}\label{d-5}
  V_{\lambda}(x^{0})<0.
\end{equation}

In terms of (v) in Theorem \ref{Th1}\,, \eqref{d-5} and the lower
semi-continuity of $V_{\lambda}$ on $\overline{\Omega}$, we can show
that there exists $x^{1}\in\Omega$ such that
\begin{equation*}
   V_{\lambda}(x^{1})=\min_{\Omega} V_{\lambda}<0
\end{equation*}
for sufficiently large $|x^{1}|$. By proceeding similarly as
\eqref{d-3}, we have
\begin{equation}\label{d-6}
  (-\Delta)_{p}^{t}v_{\lambda}(x^{1})-(-\Delta)_{p}^{t}v(x^{1})\leq C \,\frac{V_{\lambda}(x^{1})}{|x^{1}|^{\tau(p-2)+tp}}.
\end{equation}

Finally, utilizing the second inequality in \eqref{decay-model}, \eqref{d-6}, $C_{2}(x),\,C_{3}(x)<0$, $C_{4}(x)\geq0$, \eqref{d-4} and \eqref{decay-C23}, we can conclude a contradiction as follows
\begin{eqnarray*}
  0 &\leq& \left(-\Delta\right)_{p}^{t}v_{\lambda}(x^{1})-\left(-\Delta\right)_{p}^{t}v(x^{1})+C_{3}(x^{1})U_{\lambda}(x^{1})
  +C_{4}(x^{1})V_{\lambda}(x^{1}) \\
   &\leq& C \,\frac{V_{\lambda}(x^{1})}{|x^{1}|^{\tau(p-2)+tp}}+ C_{3}(x^{1})U_{\lambda}(x^{0}) \\
   &\leq&C \,\frac{V_{\lambda}(x^{1})}{|x^{1}|^{\tau(p-2)+tp}}-CC_{3}(x^{1})C_{2}(x^{0})|x^{0}|^{\gamma(p-2)+sp}V_{\lambda}(x^{1}) \\
  &=&  \frac{ V_{\lambda}(x^{1})}{|x^{1}|^{\tau(p-2)+tp}}\left(C
  -CC_{2}(x^{0})|x^{0}|^{\gamma(p-2)+sp}C_{3}(x^{1})|x^{1}|^{\tau(p-2)+tp}\right)\\
  &<&0.
\end{eqnarray*}
for sufficiently large $|x^{0}|$ and $|x^{1}|$. Hence, the relation \eqref{decay-result} must be valid for at least one of $x^{0}$ and $x^{1}$. The proof of Theorem \ref{Decay} is completed.
\end{proof}

\begin{remark}
We believe that Theorem \ref{NRP}\,, \ref{Decay} and the arguments behind the proof will be applied to other nonlinear nonlocal systems with fractional $p$-Laplacian.
\end{remark}

\section{Radial Symmetry of Positive Solutions}\label{section3}
In this section, we establish the radial symmetry of positive solutions to \eqref{model1} in the whole space and the unit ball (i.e, Theorem \ref{Th1} and \ref{Th2}\,) based on the direct method of moving planes. We start by proving Theorem \ref{Th1}\,.

\begin{proof}[Proof of Theorem \ref{Th1}]Choosing a direction to be $x_{1}$-direction and keeping the notations $T_{\lambda}$, $\Sigma_{\lambda}$, $x_{\lambda}$, $U_{\lambda}$ and $V_{\lambda}$ defined in Section \ref{section2}\,, we divide the proof into two steps.

\noindent \textup{\textbf{Step 1.}} Start moving the plane $T_{\lambda}$ from $-\infty$ to the right along the $x_{1}$-axis. We first argue that the assertion
\begin{equation}\label{r-1}
  U_{\lambda}(x),\, V_{\lambda}(x)\geq0, \quad x\in\Sigma_{\lambda}
\end{equation}
is true for sufficiently negative $\lambda$.

If \eqref{r-1} is violated, without loss of generality, we assume that there exists an $x^0 \in \Sigma_\lambda$ such that
$$ U_{\lambda}(x^0) = \min_{\Sigma_\lambda} U_\lambda < 0 .$$
By proceeding similarly as \eqref{d-3}, we get
\begin{equation}\label{r-2}
  (-\Delta)_{p}^{s}u_{\lambda}(x^{0})-(-\Delta)_{p}^{s}u(x^{0})\leq C \,\frac{U_{\lambda}(x^{0})}{|x^{0}|^{\gamma(p-2)+sp}}.
\end{equation}
Now we show that
\begin{equation}\label{r-3}
V_{\lambda}(x^{0})<0.
\end{equation}
If not, applying the assumptions $a>0$, (i), (ii), (iv), (v) in Theorem \ref{Th1} and combining with mean value theorem, we obtain
\begin{eqnarray}\label{r-4}
  &&(-\Delta)_{p}^{s}u_{\lambda}(x^{0})-(-\Delta)_{p}^{s}u(x^{0}) \nonumber\\
  &=& f\left(u_{\lambda}(x^{0}),v_{\lambda}(x^{0})\right)-au_{\lambda}^{p-1}(x^{0})-\left(f\left(u(x^{0}),v(x^{0})\right)-au^{p-1}(x^{0})\right)\nonumber\\
  &=&\left(\frac{\partial f}{\partial u}(\xi_{1},v(x^{0}))-a(p-1)\xi_{1}^{p-2}\right)U_{\lambda}(x^{0})+\left(f\left(u_{\lambda}(x^{0}),v_{\lambda}(x^{0})\right)-f\left(u_{\lambda}(x^{0}),v(x^{0})\right)\right)
  \nonumber\\
  &\geq&\left(\frac{\partial f}{\partial u}(u(x^{0}),v(x^{0}))-a(p-1)u^{p-2}(x^{0})\right)U_{\lambda}(x^{0})\nonumber\\
  &\geq&u^{m-1}(x^{0})v^{n}(x^{0})U_{\lambda}(x^{0})\nonumber\\
  &\geq&\frac{C}{|x^{0}|^{\gamma(m-1)+\tau n}}U_{\lambda}(x^{0})
\end{eqnarray}
for sufficiently negative $\lambda$, where $\xi_{1}\in\left(u_{\lambda}(x^{0}),\,u(x^{0})\right)$. Note that \eqref{r-4} contradicts with \eqref{r-2}, which is ensured by $\gamma(m-1)+\tau n>\gamma(p-2)+sp$. Thus, \eqref{r-3} holds. In terms of \eqref{r-3} and (v), we can conclude there exists an $x^1 \in \Sigma_\lambda$ such that
$$ V_{\lambda}(x^1) = \min_{\Sigma_\lambda} V_\lambda < 0 .$$
In analogy with the above argument, then (i), (iii), (iv), (v) and $\tau(r-1)+\gamma q>\tau(p-2)+tp$ are necessary to guarantee the validity of
$U_{\lambda}(x^{1})<0$.

Thus, in terms of the above estimates, mean value theorem, (ii), (iii) and (iv), we can derive
\begin{eqnarray}\label{r-6}
   &&(-\Delta)_{p}^{s}u_{\lambda}(x^{0})-(-\Delta)_{p}^{s}u(x^{0})\nonumber\\
    &=&
  \left(\frac{\partial f}{\partial u}(\xi_{1},v(x^{0}))-a(p-1)\xi_{1}^{p-2}\right)U_{\lambda}(x^{0})+\frac{\partial f}{\partial v}(u_{\lambda}(x^{0}),\eta_{1})V_{\lambda}(x^{0})\nonumber\\
  &\geq& \left(u^{m-1}(x^{0})v^{n}(x^{0})-a(p-1)u^{p-2}(x^{0})\right)U_{\lambda}(x^{0})+u^{m}(x^{0})v^{n-1}(x^{0})V_{\lambda}(x^{0}),
\end{eqnarray}
and
\begin{eqnarray}\label{r-7}
  &&(-\Delta)_{p}^{t}v_{\lambda}(x^{1})-(-\Delta)_{p}^{t}v(x^{1})\nonumber\\
  & =&
  \frac{\partial g}{\partial u}(\xi_{2},v_{\lambda}(x^{1}))U_{\lambda}(x^{1})+\left(\frac{\partial g}{\partial v}(u(x^{1}),\eta_{2})-b(p-1)\eta_{2}^{p-2}\right)
V_{\lambda}(x^{1})\nonumber\\
&\geq&u^{q-1}(x^{1})v^{r}(x^{1})U_{\lambda}(x^{1})+\left(u^{q}(x^{1})v^{r-1}(x^{1})-b(p-1)v^{p-2}(x^{1})\right)V_{\lambda}(x^{1}),
\end{eqnarray}
where $\xi_{1}\in\left(u_{\lambda}(x^{0}),\,u(x^{0})\right)$, $\eta_{1}\in\left(v_{\lambda}(x^{0}),\,v(x^{0})\right)$, $\xi_{2}\in\left(u_{\lambda}(x^{1}),\,u(x^{1})\right)$ and $\eta_{2}\in\left(v_{\lambda}(x^{1}),\,v(x^{1})\right)$, respectively.
Let
\begin{eqnarray*}
  C_{1}(x^{0}) &=&a(p-1)u^{p-2}(x^{0})-u^{m-1}(x^{0})v^{n}(x^{0})\\
      &\sim& \frac{a(p-1)}{|x^{0}|^{\gamma (p-2)}}-\frac{1}{|x^{0}|^{\gamma (m-1)+\tau n}}\geq0,
\end{eqnarray*}
\begin{eqnarray*}
  0>C_{2}(x^{0}) &=& -u^{m}(x^{0})v^{n-1}(x^{0})\\
    &\sim& \frac{1}{|x^{0}|^{\gamma m+\tau(n-1)}},
\end{eqnarray*}
\begin{eqnarray*}
  0>C_{3}(x^{1}) &=& -u^{q-1}(x^{1})v^{r}(x^{1})\\
   &\sim& \frac{1}{|x^{1}|^{\gamma(q-1)+\tau r}},
\end{eqnarray*}
and
\begin{eqnarray*}
 C_{4}(x^{1}) &=& b(p-1)v^{p-2}(x^{1})-u^{q}(x^{1})v^{r-1}(x^{1})\\
   &\sim& \frac{b(p-1)}{|x^{1}|^{\tau (p-2)}}-\frac{1}{|x^{1}|^{\gamma q+\tau (r-1)}}\geq0,
\end{eqnarray*}
for sufficiently negative $\lambda$, which are ensured by $a,\,b>0$, $p>2$, $m,\,r\geq p-1$ and (v). Then by virtue of the proof of Theorem \ref{Decay}\,, \eqref{parameter1} and \eqref{parameter2}, it implies that one of $U_{\lambda}(x)$ and $V_{\lambda}(x)$ must be nonnegative in $\Sigma_{\lambda}$ for sufficiently negative $\lambda$. Without loss of generality, we can suppose that
\begin{equation}\label{r-8}
  U_{\lambda}(x)\geq 0, \quad x\in\Sigma_{\lambda}.
\end{equation}
To show that \eqref{r-8} also holds for $V_{\lambda}(x)$, we argue by contradiction again. If $V_{\lambda}(x)$ is negative at some point in $\Sigma_{\lambda}$, then (v) guarantees there exists an $x^1 \in \Sigma_\lambda$ such that
$$ V_{\lambda}(x^1) = \min_{\Sigma_\lambda} V_\lambda < 0 .$$
From \eqref{d-6} and the similar argument as \eqref{r-4}, we derive
\begin{equation*}
  \frac{C V_{\lambda}(x^{1})}{|x^{1}|^{\tau(p-2)+tp}}\geq (-\Delta)_{p}^{t}v_{\lambda}(x^{1})-(-\Delta)_{p}^{t}v(x^{1})
  \geq \frac{C V_{\lambda}(x^{1})}{|x^{1}|^{\gamma q+\tau (r-1)}},
\end{equation*}
then $\tau(p-2)+tp<\gamma q+\tau (r-1)$ deduces a contradiction for sufficiently negative $\lambda$. Hence, \eqref{r-1} is true, which provides a starting point to move the plane $T_{\lambda}$.

\noindent \textup{\textbf{Step 2.}} Continue to move the plane $T_{\lambda}$ to the right along the $x_{1}$-axis as long as \eqref{r-1} holds
 to its limiting position . More precisely, let
 \begin{equation*}
 \lambda_0 := \sup \{ \lambda \mid U_\mu (x) \geq 0,\, V_\mu (x) \geq 0, \; x \in \Sigma_\mu, \; \mu \leq \lambda \},
 \end{equation*}
 then the behavior of $u$ and $v$ at infinity guarantee $\lambda_{0}<\infty$.

Next, we claim that $u$ is symmetric about the limiting plane $T_{\lambda_0}$, that is to say
\begin{equation}\label{r-9}
 U_{\lambda_0}(x)= V_{\lambda_0}(x)  \equiv 0 , \;\; x \in \mathbb{R}^{N}.
 \end{equation}
By the definition of $\lambda_{0}$, we first show that either
$$U_{\lambda_0}(x)= V_{\lambda_0}(x)  \equiv 0 , \;\; x \in \Sigma_{\lambda_0},$$
or
$$U_{\lambda_0}(x),\,V_{\lambda_0}(x) > 0 , \;\; x \in \Sigma_{\lambda_0}.$$
To prove this, without loss of generality,
we assume there a point $\widetilde{x}\in\Sigma_{\lambda_{0}}$ such that
$$U_{\lambda_{0}}(\widetilde{x})= \min_{\Sigma_{\lambda_{0}}} U_{\lambda_{0}}=0,$$
then it must be revealed that
$$ U_{\lambda_0}(x) \equiv 0 , \;\; x \in \Sigma_{\lambda_0}.$$
If not, on one hand
\begin{eqnarray*}
  &&(-\Delta)_{p}^{s}u_{\lambda_{0}}(\widetilde{x})-(-\Delta)_{p}^{s}u(\widetilde{x})\nonumber\\
   &=&  C_{N,sp}\,PV\int_{\mathbb{R}^{N}}\frac{G(u_{\lambda_{0}}(\widetilde{x})-u_{\lambda_{0}}(y))-G(u(\widetilde{x})-u(y))
   }{|\widetilde{x}-y|^{N+sp}}dy \nonumber\\
   &=& C_{N,sp}\,PV\int_{\mathbb{R}^{N}}\frac{G(u(\widetilde{x})-u_{\lambda_{0}}(y))-G(u(\widetilde{x})-u(y))
   }{|\widetilde{x}-y|^{N+sp}}dy\nonumber\\
   &=& C_{N,sp}\,PV\int_{\Sigma_{\lambda_{0}}}\left[\frac{1}{|\widetilde{x}-y|^{N+sp}}-\frac{1}{|\widetilde{x}-y^{\lambda_{0}}|^{N+sp}}\right]
   \left[G(u(\widetilde{x})-u_{\lambda_{0}}(y))-G(u(\widetilde{x})-u(y))\right]dy \nonumber\\
   &<&0.
\end{eqnarray*}
On the other hand,
\begin{equation*}
   (-\Delta)_{p}^{s}u_{\lambda_{0}}(\widetilde{x})-(-\Delta)_{p}^{s}u(\widetilde{x}) = \frac{\partial f}{\partial v}(u_{\lambda_{0}}(\widetilde{x}),\eta_{1})V_{\lambda_{0}}(\widetilde{x})\geq0,
\end{equation*}
which deduces a contradiction. Then it follows from the anti-symmetry of $U_{\lambda}$ that
$$ U_{\lambda_0}(x) \equiv 0 , \;\; x \in \mathbb{R}^{n},$$
which can deduce $V_{\lambda_0}(\widetilde{x})=0$. In analogy with
the above estimates, we can also derive
$$ V_{\lambda_0}(x) \equiv 0 , \;\; x \in \mathbb{R}^{n}.$$
Therefore, if \eqref{r-9} is violated, then we only have the case that
\begin{equation}\label{r-10}
U_{\lambda_0}(x),\,V_{\lambda_0}(x) > 0 , \;\; x \in \Sigma_{\lambda_0}.
\end{equation}

In the sequel, we prove that the plane can still move further in this case. To be more rigorous, there exists $\varepsilon>0$ such that
\begin{equation}\label{r-11}
  U_{\lambda}(x),\,V_{\lambda}(x) \geq 0 , \;\; x \in \Sigma_{\lambda}
\end{equation}
for any $\lambda\in \left(\lambda_{0},\lambda_{0}+\varepsilon\right)$. This is a contradiction with the definition of $\lambda_{0}$, then \eqref{r-9} holds.

Now we prove the assertion \eqref{r-11}.
From \eqref{r-10}, we have the following bounded away from zero estimate
\begin{equation*}
  U_{\lambda_{0}}(x),\,V_{\lambda_{0}}(x) \geq C_{\delta}>0 , \;\; x \in \Sigma_{\lambda_{0}-\delta}\cap B_{R_{0}}(0)
\end{equation*}
for some $R_{0}>0$. By the continuity of $U_{\lambda}(x)$ and $V_{\lambda}(x)$ with respect to $\lambda$, there exists a positive constant $\varepsilon$ such that
\begin{equation*}
  U_{\lambda}(x),\,V_{\lambda}(x) \geq 0 , \;\; x \in \Sigma_{\lambda_{0}-\delta}\cap B_{R_{0}}(0)
\end{equation*}
for any $\lambda\in \left(\lambda_{0},\lambda_{0}+\varepsilon\right)$.
Moreover, by virtue of Theorem \ref{Decay} , we know that if
 $$U_{\lambda}(x^0) = \min_{\Sigma_\lambda} U_\lambda < 0 \quad \mbox{and} \quad V_{\lambda}(x^1) = \min_{\Sigma_\lambda} V_\lambda < 0,$$
 then there exists a positive constant $R_{0}$ large enough such that one of $x^{0}$ and $x^{1}$ must be in $B_{R_{0}}(0)$. We may as well suppose $|x^{0}|<R_{0}$. Thus, we obtain
\begin{equation}\label{r-12}
  x^{0}\in\left(\Sigma_{\lambda}\backslash\Sigma_{\lambda_{0}-\delta}\right)\cap B_{R_{0}}(0).
\end{equation}
Next, we show that \eqref{r-12} also holds for $x^{1}$. If $x^{1}\in \Sigma_{\lambda}\cap B^{C}_{R_{0}}(0)$, then by virtue of \eqref{d-6}, \eqref{r-6}, \eqref{r-7}, (i), (iii), (iv), (v), $b>0$, $p>2$, $r\geq p-1$ and \eqref{n-8}, we have
\begin{eqnarray}\label{r-14}
 \frac{C V_{\lambda}(x^{1})}{|x^{1}|^{\tau(p-2)+tp}}&\geq& (-\Delta)_{p}^{t}v_{\lambda}(x^{1})-(-\Delta)_{p}^{t}v(x^{1})\nonumber\\
 &=&  \frac{\partial g}{\partial u}(\xi_{2},v_{\lambda}(x^{1}))U_{\lambda}(x^{1})+\left(\frac{\partial g}{\partial v}(u(x^{1}),\eta_{2})-b(p-1)\eta_{2}^{p-2}\right)
V_{\lambda}(x^{1})\nonumber\\
 &\geq&\frac{\partial g}{\partial u}(\xi_{2},v_{\lambda}(x^{1}))U_{\lambda}(x^{0})+\left(u^{q}(x^{1})v^{r-1}(x^{1})-b(p-1)v^{p-2}(x^{1})\right)V_{\lambda}(x^{1})\nonumber\\
  &\geq&u^{q-1}(x^{1})v^{r}(x^{1})U_{\lambda}(x^{0})+\left(\frac{C}{|x^{1}|^{\gamma q+\tau(r-1)}}-\frac{Cb(p-1)}{|x^{1}|^{(p-2)\tau}}\right)V_{\lambda}(x^{1})\nonumber\\
       &\geq&\frac{C}{|x^{1}|^{\gamma(q-1)+\tau r}}U_{\lambda}(x^{0})
\end{eqnarray}
and
\begin{eqnarray}\label{r-15}
  -C\delta_{x^{0}}&\geq&  (-\Delta)_{p}^{s}u_{\lambda}(x^{0})-(-\Delta)_{p}^{s}u(x^{0}) \nonumber\\
   &=&  \left(\frac{\partial f}{\partial u}(\xi_{1},v(x^{0}))-a(p-1)\xi_{1}^{p-2}\right)U_{\lambda}(x^{0})+\frac{\partial f}{\partial v}(u_{\lambda}(x^{0}),\eta)V_{\lambda}(x^{0})
\end{eqnarray}
for sufficiently small $\delta$ and $\varepsilon$ and large $R_{0}$,
where $\xi_{1}\in\left(u_{\lambda}(x^{0}),\,u(x^{0})\right)$,
$\xi_{2}\in\left(u_{\lambda}(x^{1}),\,u(x^{1})\right)$ and $\eta\in
\left(v_{\lambda}(x^{0}),v(x^{0})\right)$. Hence, utilizing the
above inequalities, (i) and \eqref{n-10}, we derive
\begin{eqnarray}\label{r-13}
  1&\leq& -\frac{C}{\delta_{x^{0}}}\left[\left(\frac{\partial f}{\partial u}(\xi_{1},v(x^{0}))-a(p-1)\xi_{1}^{p-2}\right)U_{\lambda}(x^{0})+ \frac{\partial f}{\partial v}(u_{\lambda}(x^{0}),\eta)V_{\lambda}(x^{1})\right]\nonumber \\
  &\leq& -\frac{C}{\delta_{x^{0}}}\left[\left(\frac{\partial f}{\partial u}(\xi_{1},v(x^{0}))-a(p-1)\xi_{1}^{p-2}\right)U_{\lambda}(x^{0})+ C\frac{|x^{1}|^{\tau(p-2)+tp}}{|x^{1}|^{\gamma(q-1)+\tau r}}\frac{\partial f}{\partial v}(u_{\lambda}(x^{0}),\eta)U_{\lambda}(x^{0})\right]\nonumber \\
   &=& -\frac{C}{\delta_{x^{0}}}U_{\lambda}(x^{0})\left[\left(\frac{\partial f}{\partial u}(\xi_{1},v(x^{0}))-a(p-1)\xi_{1}^{p-2}\right)+ C\frac{\partial f}{\partial v}(u_{\lambda}(x^{0}),\eta)\frac{|x^{1}|^{\tau(p-2)+tp}}{|x^{1}|^{\gamma(q-1)+\tau r}}\right] \nonumber\\
   &\leq& C \,o(1) \left[\left(\frac{\partial f}{\partial u}(\xi_{1},v(x^{0}))-a(p-1)\xi_{1}^{p-2}\right)+ \frac{\partial f}{\partial v}(u_{\lambda}(x^{0}),\eta)\frac{|x^{1}|^{\tau(p-2)+tp}}{|x^{1}|^{\gamma(q-1)+\tau r}}\right]
\end{eqnarray}
for sufficiently small $\delta$, $\varepsilon$ and large $R_{0}$.  Note that
$$\left(\frac{\partial f}{\partial u}(\xi_{1},v(x^{0}))-a(p-1)\xi_{1}^{p-2}\right)+ \frac{\partial f}{\partial v}(u_{\lambda}(x^{0}),\eta)\frac{|x^{1}|^{\tau(p-2)+tp}}{|x^{1}|^{\gamma(q-1)+\tau r}}$$
is bounded, which is ensured by $ \tau r+\gamma(q-1)>\tau(p-2)+tp$, $|x^{1}|>R_{0}$, $x^{0}\in\left(\Sigma_{\lambda}\backslash\Sigma_{\lambda_{0}-\delta}\right)\cap B_{R_{0}}(0)$ and $f\in C^{1}$. Hence, \eqref{r-13} must not be valid for sufficiently small $\delta$ and $\varepsilon$. This contradiction deduces that
 $$x^{1}\in\left(\Sigma_{\lambda}\backslash\Sigma_{\lambda_{0}-\delta}\right)\cap B_{R_{0}}(0).$$
In terms of Theorem \ref{NRP}\,, we can conclude that
\begin{equation*}
  U_{\lambda}(x),\,V_{\lambda}(x) \geq 0 , \;\; x \in \left(\Sigma_{\lambda}\backslash\Sigma_{\lambda_{0}-\delta}\right)\cap B_{R_{0}}(0)
\end{equation*}
for sufficiently small $\delta$ and $\varepsilon$. Hence, \eqref{r-11} holds.

Therefore, the above contradiction means that
$$ U_{\lambda_0}(x)= V_{\lambda_0}(x)  \equiv 0 , \;\; x \in \Sigma_{\lambda_0}.$$
Since $x_1$ direction can be chosen arbitrarily, so we can conclude that the positive solution pair $u$ and $v$ must be radially symmetric and monotone decreasing with respect to some point in $\mathbb{R}^{n}$. This completes the proof of the Theorem \ref{Th1} .
\end{proof}

We now turn our attention to prove Theorem \ref{Th2}\,.
\begin{proof}[Proof of Theorem \ref{Th2}]Choosing a direction to be $x_{1}$-direction and start moving the plane $T_{\lambda}$ from $-1$ to the right along the $x_{1}$-axis, we proceed in two steps and first argue that the assertion
\begin{equation}\label{b-1}
  U_{\lambda}(x),\, V_{\lambda}(x)\geq0, \quad x\in\Omega_{\lambda}
\end{equation}
is true for $\lambda>-1$ sufficiently closing to $-1$, where
$\Omega_{\lambda} :=\{x \in B_{1}(0) \mid \, x_1<\lambda\}$. After a direct calculation, we have
\begin{equation}\label{b-2}
   (-\Delta)_{p}^{s}u_{\lambda}(x)-(-\Delta)_{p}^{s}u(x)+C_{1}(x)U_{\lambda}(x)+C_{2}(x)V_{\lambda}(x)=0
\end{equation}
and
\begin{equation}\label{b-3}
    (-\Delta)_{p}^{t}v_{\lambda}(x)-(-\Delta)_{p}^{t}v(x)+C_{3}(x)U_{\lambda}(x)+C_{4}(x)V_{\lambda}(x)=0,
\end{equation}
where
\begin{eqnarray*}
  C_{1}(x) &=& a(p-1)\xi^{p-2}-\frac{f(u_{\lambda}(x),v(x))-f(u(x),v(x))}{u_{\lambda}(x)-u(x)} \\
  &\geq&-\frac{f(u_{\lambda}(x),v(x))-f(u(x),v(x))}{u_{\lambda}(x)-u(x)},\\
  C_{2}(x) &=& -\frac{f(u_{\lambda}(x),v_{\lambda}(x))-f(u_{\lambda}(x),v(x))}{v_{\lambda}(x)-v(x)}, \\
  C_{3}(x) &=& -\frac{g(u_{\lambda}(x),v_{\lambda}(x))-g(u(x),v_{\lambda}(x))}{u_{\lambda}(x)-u(x)}, \\
  C_{4}(x) &=& b(p-1)\eta^{p-2}-\frac{g(u(x),v_{\lambda}(x))-g(u(x),v(x))}{v_{\lambda}(x)-v(x)} \\
  &\geq& -\frac{g(u(x),v_{\lambda}(x))-g(u(x),v(x))}{v_{\lambda}(x)-v(x)}.
\end{eqnarray*}
for $U_{\lambda}(x),\,V_{\lambda}(x)\neq0$. Here $\xi$ is between
$u(x)$ and $u_{\lambda}(x)$, $\eta$ is between $v(x)$ and
$v_{\lambda}(x)$. Applying the assumptions that $f,\,g$ are
Lipschitz continuous and combining \eqref{f} with \eqref{g}, we show
that $C_{1}(x)$, $C_{2}(x)$, $C_{3}(x)$, $C_{4}(x)$ have lower
bounds and $C_{2}(x),\,C_{3}(x)<0$ in $\Omega_{\lambda}$. Besides, a
combination of $u,\,v>0$ on $B_{1}(0)$ and $u,\,v\equiv0$ on
$\mathbb{R}^{N} \backslash B_{1}(0)$ yields that the additional
conditions in Theorem \ref{NRP} are automatically satisfied. Hence,
in terms of Theorem \ref{NRP}\,, we conclude the assertion
\eqref{b-1} holds for $\lambda>-1$ sufficiently closing to $-1$.

Next, we continue to move the plane $T_{\lambda}$ to the right along
the $x_{1}$-axis until its limiting position as long as \eqref{b-1}
holds. More precisely, defining
 \begin{equation*}
 \lambda_0 := \sup \{ \lambda\leq0 \mid U_\mu (x) \geq 0,\, V_\mu (x) \geq 0, \; x \in \Omega_\mu, \; \mu \leq \lambda \}.
 \end{equation*}
We now claim that
\begin{equation}\label{b-4}
  \lambda_{0}=0.
\end{equation}
If not, then we will prove that the plane can still move further such that \eqref{b-1} holds. To be more rigorous, there exists $\varepsilon>0$ such that
\begin{equation}\label{b-5}
  U_{\lambda}(x),\,V_{\lambda}(x) \geq 0 , \;\; x \in \Omega_{\lambda}
\end{equation}
for any $\lambda\in \left(\lambda_{0},\lambda_{0}+\varepsilon\right)$, which contradicts the definition of $\lambda_{0}$.
Since both $U_{\lambda_0}(x)$ and $V_{\lambda_0}(x)$ are not identically zero on $\Omega_{\lambda_{0}}$, we utilize the similar argument as in the proof of Theorem \ref{Th1} yields
\begin{equation}\label{b-6}
U_{\lambda_0}(x),\,V_{\lambda_0}(x) > 0 , \;\; x \in \Omega_{\lambda_0}.
\end{equation}
It follows from \eqref{b-6} that
\begin{equation*}
  U_{\lambda_{0}}(x),\,V_{\lambda_{0}}(x) \geq C_{\delta}>0 , \;\; x \in \Omega_{\lambda_{0}-\delta}.
\end{equation*}
Thus, by the continuity of $U_{\lambda}(x)$ and $V_{\lambda}(x)$ with respect to $\lambda$, there exists a positive constant $\varepsilon$ such that
\begin{equation*}
  U_{\lambda}(x),\,V_{\lambda}(x) \geq 0 , \;\; x \in \Omega_{\lambda_{0}-\delta}
\end{equation*}
for any $\lambda\in \left(\lambda_{0},\lambda_{0}+\varepsilon\right)$. Selecting $\Omega_{\lambda}\backslash\Omega_{\lambda_{0}-\delta}$ as a narrow region, then \eqref{b-5} holds for sufficiently small $\delta$ and $\varepsilon$ by Theorem \ref{NRP}\,. Hence, the assertion \eqref{b-4} is proved.

Finally, we conclude that the positive solution pair $u$ and $v$ are radially symmetric and monotone decreasing about the origin due to $x_1$ direction can be chosen arbitrarily. This completes the proof of the Theorem \ref{Th2} .
\end{proof}

\section{Monotonicity and Nonexistence of Positive Solutions}\label{section4}
In this section, applying the direct method of moving planes to prove Theorem \ref{Th3} and \ref{Th4}\,, we show that the monotonicity in an unbounded parabolic domain and the nonexistence on the half space for positive solutions to \eqref{model1}, respectively. We proceed by proving Theorem \ref{Th3}\,.
\begin{proof}[Proof of Theorem \ref{Th3}]
A direct calculation shows that the coefficients in \eqref{b-2} and \eqref{b-3} are replaced by
\begin{eqnarray*}
  C_{1}(x) &=& a(x')(p-1)\xi^{p-2}-\frac{f(u_{\lambda}(x),v(x))-f(u(x),v(x))}{u_{\lambda}(x)-u(x)}, \\
  C_{2}(x) &=& -\frac{f(u_{\lambda}(x),v_{\lambda}(x))-f(u_{\lambda}(x),v(x))}{v_{\lambda}(x)-v(x)}, \\
  C_{3}(x) &=& -\frac{g(u_{\lambda}(x),v_{\lambda}(x))-g(u(x),v_{\lambda}(x))}{u_{\lambda}(x)-u(x)}, \\
  C_{4}(x) &=& b(x')(p-1)\eta^{p-2}-\frac{g(u(x),v_{\lambda}(x))-g(u(x),v(x))}{v_{\lambda}(x)-v(x)}, \\
\end{eqnarray*}
for $U_{\lambda}(x),\,V_{\lambda}(x)\neq0$, where $\xi$ is between
$u(x)$ and $u_{\lambda}(x)$, $\eta$ is between $v(x)$ and
$v_{\lambda}(x)$. By virtue of the assumptions in Theorem
\ref{Th3}\,, we can apply Theorem \ref{NRP} to deduce that
\begin{equation}\label{om-1}
  U_{\lambda}(x),\, V_{\lambda}(x)\geq0, \quad x\in\widehat{\Omega}_{\lambda}
\end{equation}
for $\lambda>0$ sufficiently closing to $0$, where
$\widehat{\Omega}_{\lambda} :=\{x \in \Omega\mid \, x_N<\lambda\}$ and $x^{\lambda}:=(x',2\lambda-x_{N})$.

We continue to move the plane $\widehat{T}_{\lambda}:=\{x \in \Omega\mid \, x_N=\lambda \mbox{ for some } \lambda\in \mathbb{R}_{+}\}$ to the right along the $x_{N}$-axis as long as \eqref{om-1} holds to its limiting position. To be more precise, let
 \begin{equation*}
 \lambda_0 := \sup \{ \lambda>0 \mid U_\mu (x) \geq 0,\, V_\mu (x) \geq 0, \; x \in \widehat{\Omega}_\mu, \; \mu \leq \lambda \}.
 \end{equation*}
We now argue the assertion that
\begin{equation}\label{om-2}
  \lambda_{0}=+\infty.
\end{equation}

Otherwise, if $ \lambda_{0}<+\infty$, we claim that
\begin{equation}\label{om-3}
U_{\lambda_0}(x)= V_{\lambda_0}(x)  \equiv 0 , \;\; x \in \widehat{\Omega}_{\lambda_0}.
\end{equation}
In analogy with the proof of Theorem \ref{Th1}\,, we can derive either \eqref{om-3} or
\begin{equation}\label{om-4}
U_{\lambda_0}(x),\,V_{\lambda_0}(x)  >0  , \;\; x \in \widehat{\Omega}_{\lambda_0}
\end{equation}
holds. If \eqref{om-4} is true, then we will prove that the plane can still move further such that \eqref{om-1} holds. To be more precise, there exists $\varepsilon>0$ such that
\begin{equation}\label{om-5}
  U_{\lambda}(x),\,V_{\lambda}(x) \geq 0 , \;\; x \in \widehat{\Omega}_{\lambda}
\end{equation}
for any $\lambda\in \left(\lambda_{0},\lambda_{0}+\varepsilon\right)$, which contradicts the definition of $\lambda_{0}$.
It follows from \eqref{om-4} that
\begin{equation*}
  U_{\lambda_{0}}(x),\,V_{\lambda_{0}}(x) \geq C_{\delta}>0 , \;\; x \in \widehat{\Omega}_{\lambda_{0}-\delta}
\end{equation*}
for $0<\delta<\lambda_{0}$.
Thus, by the continuity of $U_{\lambda}(x)$ and $V_{\lambda}(x)$ with respect to $\lambda$, there exists a positive constant $\varepsilon$ such that
\begin{equation*}
  U_{\lambda}(x),\,V_{\lambda}(x) \geq 0 , \;\; x \in \widehat{\Omega}_{\lambda_{0}-\delta}
\end{equation*}
for any $\lambda\in \left(\lambda_{0},\lambda_{0}+\varepsilon\right)$. We specify $\widehat{\Omega}_{\lambda}\backslash\widehat{\Omega}_{\lambda_{0}-\delta}$ as a narrow region, then \eqref{om-5} holds for sufficiently small $\delta$ and $\varepsilon$ by Theorem \ref{NRP}\,. Hence, the aforementioned contradiction concludes that \eqref{om-3} is valid.

We mention that \eqref{om-3} implies
\begin{equation*}
  u(x_{1},x_{2},...,x_{N-1},2\lambda_{0})=u(x_{1},x_{2},...,x_{N-1},0)=0
\end{equation*}
and
\begin{equation*}
  v(x_{1},x_{2},...,x_{N-1},2\lambda_{0})=v(x_{1},x_{2},...,x_{N-1},0)=0,
\end{equation*}
which are contradictive with the fact that $u,\,v>0$ on $\Omega$,
then the assertion \eqref{om-2} holds. Therefore, $u$ and $v$ are
strictly increasing with respect to the $x_{N}$-axis, which
completes the proof of the Theorem \ref{Th3} .
\end{proof}

In the sequel, it remains to be proved Theorem \ref{Th4}\,.
\begin{proof}[Proof of Theorem \ref{Th4}] We start by proving the assertion that
\begin{equation}\label{R+1}
  \mbox{either} \,\, u(x),\, v(x)>0 \,\, \mbox{or}\,\, u(x),\, v(x)\equiv0 \,\, \mbox{in} \,\, \mathbb{R}_{+}^{N}.
\end{equation}
We first show that if there exists $x_{0}\in\mathbb{R}_{+}^{N}$ such that $u(x_{0})=0$, then
\begin{equation}\label{R+2}
  u(x),\, v(x)\equiv0 \,\, \mbox{in} \,\, \mathbb{R}_{+}^{N}.
\end{equation}
If $u(x)\not\equiv0$, on one hand
\begin{eqnarray*}
  (-\Delta)_p^s u(x_{0}) &=& C_{N,sp} PV \int_{\mathbb{R}^N} \frac{|u(x_{0})-u(y)|^{p-2}[u(x_{0})-u(y)]}{|x_{0}-y|^{N+sp}} dy\\
   &=&C_{N,sp} PV \int_{\mathbb{R}_{+}^N} \frac{-|u(y)|^{p-2}u(y)}{|x_{0}-y|^{N+sp}} dy  \\
  &<&  0.
\end{eqnarray*}
On the other hand, it follows from \eqref{model1-R+N}, \eqref{R+-fg} and \eqref{f} that
\begin{equation*}
  (-\Delta)_p^s u(x_{0})  = f(0,v(x_{0})) \geq f(0,0)=0.
\end{equation*}
This contradiction implies that $u(x)\equiv0$ in $\mathbb{R}_{+}^{N}$. Then we have $f(0,v)=0$, which is ensured by $u(x)\equiv0$ and the first equation in \eqref{model1-R+N}. Now using \eqref{R+-fg} and \eqref{f} again, we can deduce that $v(x)\equiv0$ on $\mathbb{R}_{+}^{N}$. Indeed, the similar argument as in the proof of \eqref{R+2} yields if $v(x)$ attains zero at a point in $\mathbb{R}_{+}^{N}$, then
$u(x),\, v(x)\equiv0$ in $\mathbb{R}_{+}^{N}$. Hence, the assertion \eqref{R+1} holds.

Now we prove Theorem \ref{Th4} by contradiction. In the sequel, we
always assume that
\begin{equation}\label{R+3}
  u(x),\, v(x)>0 \,\, \mbox{in} \,\, \mathbb{R}_{+}^{N}.
\end{equation}
Adopting the notations
 $$T'_{\lambda} :=\{x \in \mathbb{R}_{+}^{N}\mid x_N=\lambda \mbox{ for some } \lambda\in \mathbb{R}_{+}\},$$
  $$\Sigma'_{\lambda} :=\{x \in \mathbb{R}_{+}^{N} \mid x_N<\lambda\},$$
 and denoting the reflection of $x$ about the moving plane $T'_{\lambda}$ by $ x^{\lambda} :=(x_1, x_2, ..., 2\lambda-x_N)$.
We proceed in two steps and first argue
\begin{equation}\label{R+4}
  U_{\lambda}(x),\, V_{\lambda}(x)\geq0, \quad x\in\Sigma'_{\lambda}
\end{equation}
is valid for $\lambda>0$ sufficiently closing to $0$. A combination of \eqref{R+3} and \eqref{decay-uv} yields that
\begin{equation}\label{R+5}
   \underset{|x| \rightarrow \infty}{\underline{\lim}}U_{\lambda}(x),\,\underset{|x| \rightarrow \infty}{\underline{\lim}}V_{\lambda}(x)\geq0.
\end{equation}
Thus, we conclude the assertion \eqref{R+4} by Theorem \ref{NRP}\,.

Next, we continue to move the plane $T'_{\lambda}$ to the right along the $x_{N}$-axis until its limiting position as long as \eqref{R+4} holds. More precisely, let
 \begin{equation*}
 \lambda_0 := \sup \{ \lambda>0 \mid U_\mu (x) \geq 0,\, V_\mu (x) \geq 0, \; x \in \Sigma'_\mu, \; \mu \leq \lambda \}.
 \end{equation*}
We show that
\begin{equation}\label{R+6}
  \lambda_{0}=+\infty.
\end{equation}
Otherwise, if $\lambda_{0}<+\infty$, combining \eqref{R+5} with the similar argument as in the proof of \eqref{om-3},
 we can deduce
$$U_{\lambda_0}(x)= V_{\lambda_0}(x)  \equiv 0 , \;\; x \in \Sigma'_{\lambda_0},$$
It reveals that
\begin{equation*}
  u(x_{1},x_{2},...,x_{N-1},2\lambda_{0})=u(x_{1},x_{2},...,x_{N-1},0)=0
\end{equation*}
and
\begin{equation*}
  v(x_{1},x_{2},...,x_{N-1},2\lambda_{0})=v(x_{1},x_{2},...,x_{N-1},0)=0,
\end{equation*}
which are contradictive with the assumption \eqref{R+3}, then \eqref{R+6} holds.

Hence, $u$ and $v$ are increasing with respect to the $x_{N}$-axis. In terms of \eqref{decay-uv}, we know that this is impossible, and then  $u(x),\, v(x)\equiv0$ in  $\mathbb{R}^{N}$. This completes the proof of Theorem \ref{Th4}\,.
\end{proof}

\section*{Acknowledgments} This work is supported by the National Natural Science Foundation of China (NNSF Grant No. 11671414 and No. 11771218). The authors would like to deeply thank to Professor W. Chen for introducing this topic and providing many constructive comments.

\bibliography{bibliography}

\end{document}